\documentclass[11pt]{amsart}
\usepackage{amssymb}
\usepackage{url}
\usepackage{hyperref}
\usepackage{xcolor}
\usepackage{graphicx}
\usepackage{mathrsfs}
\usepackage{comment}
\usepackage{relsize}

\newcommand*\oline[1]{%
  \vbox{%
    \hrule height .45pt%                 
    \kern0.35ex%                         
    \hbox{%
    \kern0.05em%                        
\ifmmode#1\else\ensuremath{#1}\fi%  
    \kern0.05em%                        
    }
  }
}

\newtheorem{theorem}{Theorem}[section]
\newtheorem{lemma}[theorem]{Lemma}

\newtheorem{proposition}[theorem]{Proposition}
\newtheorem{corollary}[theorem]{Corollary}
\theoremstyle{definition}

\theoremstyle{remark}
\newtheorem{remark}[theorem]{Remark}

\newcommand{\F}{\mathbb{F}}

\newcommand{\Z}{\mathbb{Z}}
\newcommand{\Q}{\mathbb{Q}}
\renewcommand{\P}{\mathbb{P}}
\newcommand{\C}{\mathbb{C}}

\newcommand{\G}{{\Gamma}}

\newcommand{\GL}{\mathrm{GL}}

\newcommand{\SL}{\mathrm{SL}}
\newcommand{\PSL}{\mathrm{PSL}}
\newcommand{\Sp}{\mathrm{Sp}}

\newcommand{\tr}{\mathrm{tr}}

\newcommand{\Ad}{\mathrm{Ad}}

\newcommand{\calc}{\mathscr{C}}
\newcommand{\cals}{\mathscr{S}}

\begin{document}

\title[Experimenting with Zariski dense matrix groups
over number fields]
{Algorithms for experimenting with Zariski dense matrix groups
over number fields}

\author{A.~S.~Detinko, D.~L.~Flannery, A.~Hulpke}

\makeatletter
\@namedef{subjclassname@2020}{\textup{2020} 
Mathematics Subject Classification}
\makeatother


\begin{abstract}
Let $\P$ be an algebraic number field. We provide a computational
 analog of the strong approximation theorem for finitely
generated Zariski dense groups $H\leq \SL(n,\P)$, $n$ prime.
That is, we present algorithms to find the set of congruence 
quotients of $H$ modulo all maximal ideals of a finitely
generated subring $R$ of $\P$ such that $H\leq \SL(n,R)$.
 The algorithms have been implemented in {\sf GAP}. 
 Potential applications are illustrated by a range of 
experiments in degree $2$, with a special focus 
on Bianchi groups.
\end{abstract}

\maketitle

\vspace{-20pt}

\section{Introduction}

This work is a further contribution to an expanding 
area of computational group theory; namely, computing 
with matrix groups over an infinite field $\F$. Our algorithms
take as input a finite subset of $\GL(n,\F)$. This generates 
a subgroup $H$ of $\GL(n,\F)$ such that $H\leq \GL(n, R)$ for 
some finitely generated integral domain $R\subseteq \F$. 
Since $H$ is residually finite~\cite[4.2]{Wehrfritz},
 it can be investigated efficiently via its finite quotients, in 
particular the congruence quotients arising from reduction
 modulo maximal ideals of $R$.

 A single congruence quotient of $H$ is enough to solve some decision
problems; e.g., testing finiteness~\cite{Recog}. Deeper study of $H$
requires knowledge of a plurality of its congruence quotients. Two 
questions are foremost: can we find all congruence quotients of $H$; 
how do the congruence quotients of $H$ determine its structure and properties?

 By way of the strong approximation property for linear groups 
(see \cite[Window~9]{LubotzkySegal} and \cite{RapinchukSAT,Weis}),
 we can describe the congruence quotients of
$H$ when $H$ is a dense subgroup of $\SL(n,\C)$.
(Here and from now on, `dense' means `Zariski dense'.)
Indeed, if $\F = \Q$ then the strong approximation
property implies that dense $H$ surjects onto $\SL(n,p)$
modulo almost all rational primes $p$. The congruence 
quotients of such $H$ modulo all maximal ideals of $R$ are 
therefore known once we have computed the finite set $\Pi(H)$ 
of primes modulo which $H$ and $\SL(n, R)$ are \emph{not} congruent. 

In \cite{ExpMath,SAT} we gave practical algorithms 
to compute $\Pi(H)$ for input dense $H \leq \SL(n, \Q)$:
a computational analog of strong approximation.
A hallmark of our methods is that they transfer 
much of the computing to matrix groups 
over finite fields. In this regard, the {\sf GAP} package 
\cite{NSrecog}, which realizes a categorization due to 
Aschbacher~\cite{Aschbacher84} of maximal subgroups of classical 
groups over finite fields, is a major component of our 
implementations.

Strong approximation not only ensures that
we can compute the congruence quotients of a dense group $H$; 
it implies that $H$ is well approximated by its congruence 
quotients. This affords a framework for examining the structure
of $H$. In \cite{Density} we developed algorithms that accept
 dense $H\leq \SL(n,\Z)$ and $\Pi(H)$ as input and compute the 
arithmetic closure of $H$, i.e., the intersection of all 
finite-index subgroups of $\SL(n, \Z)$ containing $H$. These 
algorithms facilitated the solution of problems for which no 
methods were previously available (e.g., 
see \cite[Section~4]{Density}).

Let $\P$ be an algebraic number field. Our primary goal in this 
paper is to extend the techniques of \cite{ExpMath,SAT} to handle 
finitely generated dense subgroups $H$ of $\SL(n, \P)$. 
Now $H\leq \SL(n,R)$ for $R\subseteq \mathcal{O}[1/\mu]=
\{ a/\mu^i \, | \, a \in \mathcal{O}, i \in \Z\}$ where 
$\mathcal{O} = \mathcal{O}_{\P}$ is the ring of integers of $\P$ 
and $\mu\in \Z$, $\mu >0$. Specifically, we obtain practical 
algorithms to describe the set of congruence quotients of $H$
modulo all maximal ideals $I$ of $R$. As with the algorithms over 
$\P=\Q$, our approach is based on reduction to computing in 
$\SL(n,R/I)$. So we continue to rely on Aschbacher's categorization 
of the maximal subgroups of $\SL(n, p^k)$, $p$ prime; see 
\cite{Brayetal} for an exhaustive treatment. Suppose that $\SL(n,R/I)$ 
and the congruence image $\widetilde{H}$ of $H$ modulo a maximal ideal 
$I$ of $R$ are different. Then $\widetilde{H}$ lies in a maximal 
subgroup of $\SL(n,R/I)$ belonging to some Aschbacher class $\calc_i$, 
$1 \leq i \leq 9$. The idea is to say, for each $i$, precisely when 
there can be a maximal subgroup in $\calc_i$ containing a congruence 
image of $H$.

We emphasize that the situation for groups over $\P$
 differs markedly from that for groups over $\Q$.
Firstly, if $\P \neq \Q$ then the image
field $R/I$ is not necessarily of prime size, and
 classes of maximal subgroups of $\SL(n,R/I)$
 that are not present in $\SL(n,p)$ emerge.
Another issue is that there might be infinitely 
many maximal ideals $I$ of $R\subseteq \mathcal{O}[1/\mu]$ 
such that dense $H\leq \SL(n,R)$ does not surject onto 
$\SL(n, R/I)$ modulo $I$. 
However, under certain conditions, $H$ \emph{will} surject onto 
$\SL(n, R/I)$ modulo all but finitely many maximal  ideals $I$. We 
give algorithms to decide whether $H$ satisfies those conditions; 
if it does, then we compute the (finite) set of maximal ideals 
$I\subset R$ such that $H$ does not surject onto $\SL(n, R/I)$.

Some of our procedures are valid for any degree
$n$. Others are restricted to prime $n$, taking advantage of 
the simplification of linear group structure in prime degree.
At the same time this suffices to cover several important 
applications; cf.~\cite[Section~4]{ExpMath} and see below.
Our approach for arbitrary $n$ entails computing with adjoint 
representations, and is dealt with in \cite{SLnO}.

We are motivated by applications---e.g., in algebraic number 
theory and geometry---that feature linear groups over number 
fields other than $\Q$, especially in small degrees. As a 
sample application of our algorithms, we consider subgroups $H$ of 
Bianchi groups; for these $n=2$ and 
$\P = \Q(\sqrt{-d}\hspace{.75pt})$, where $d\in \Z$ is positive 
and square-free. If $H$ is dense then we compute $\Pi(H)$ and 
test whether $H$ has finite index. Subsequently we decide
whether each finite-index $H$ is a congruence subgroup, i.e., contains 
the full kernel of a congruence homomorphism defined on 
$\SL(2, \mathcal{O}_{\P})$. This enables us to classify 
all subgroups of $\SL(2, \Z[\omega])$ of index up to $30$, for 
a primitive cube root of unity $\omega$, and find the congruence subgroups 
among them. We repeat the latter exercise to classify all normal 
subgroups of $\SL(2, \Z[\omega])$ of index up to $20000$. Our 
algorithms are the first to be developed for solving such problems.

The paper is organized as follows. In Section~\ref{CongHom} 
we explain how to compute congruence homomorphisms defined on 
 finitely generated subgroups of $\GL(n,\P)$. In Section~\ref{Asch}
 we determine when a congruence image of $H\leq \SL(n,\P)$ 
is in a subgroup from each  Aschbacher class.
In Section~\ref{StrongAT} we use a characterization 
of dense $H$ in terms of its congruence images to design an algorithm 
for computing congruence quotients of $H$ modulo maximal ideals 
of an ambient coefficient ring. Finally, in
Section~\ref{Exper} we discuss computer experiments with 
subgroups of Bianchi groups performed using a {\sf GAP} implementation of
the algorithms, that demonstrate the effectiveness 
and applicability of our methods.

\section{Constructing congruence homomorphisms} 
\label{CongHom}

We now outline how to construct congruence homomorphisms over $\P$. 
These reduce modulo (maximal) ideals of the ground domain $R$, so that
much of our computing is with matrix groups over finite fields.
Some of the material of this section may be found also in 
\cite[Section~3.2]{Recog}.

There exists $\alpha\in \mathcal{O}$ such that  
$\mathbb{P} =  \Q[\alpha] =  \Q(\alpha)$~\cite[p.~81]{Koch}.
Let $f(\mathrm{x})\in \Z[\mathrm{x}]$ be the minimal polynomial 
of $\alpha$; so $f(\mathrm{x})$ has degree $m:= |\P:\Q|$.  

As noted, if $H=\langle X\rangle$ for finite $X\subseteq \SL(n,\P)$, 
then $H\leq \SL(n,R_H)$ where $R_H$ is a finitely generated subring 
of $\P$. We can take $R_H$ to be generated by the entries of the 
elements of $X\cup X^{-1}$.  Hence we assume that $R_H\subseteq 
\allowbreak \Z[\alpha, 1/\mu] \subseteq \allowbreak R:=
\mathcal{O}[1/\mu]$ where $\mu\in \Z$, $\mu>0$. 
Although $R$ is a finitely generated Dedekind domain, 
$R_H$ need not be Dedekind. 

For an ideal $I$ of $R$, the natural epimorphism 
$R \rightarrow R/I$ extends entrywise to $\mathrm{Mat}(n,R)$. 
We denote both of these maps by $\varphi_{I}$; they are 
\emph{congruence homomorphisms} (modulo $I$).
Let $t_{i,j}(a)$ for $i\neq j$ be the 
element of $\SL(n,R)$ with $1$s down its main diagonal, $a\in R$ 
in position $(i,j)$, and zeros elsewhere.  
 For any field $\F$, $\SL(n,\F) = \langle
 t_{i,j}(a) : a\in \F\hspace{.5pt}\rangle$. 
Thus, if $I$ is a maximal ideal of $R$ then $\varphi_I$  maps 
$\SL(n, R)$ onto $\SL(n, p^k)$ for some $k \leq m$ and prime $p$ 
such that $R/I \cong \F_{p^k}$ ($\F_q$ is the field of size $q$).
Each congruence homomorphism $\varphi_I$ with image over a finite 
field restricts to one over $R_H$; the intersection of $R_H$ and 
a maximal ideal $I$ of $R$ is a maximal ideal of $R_H$. 

For a prime $p$ and $g(\mathrm{x})\in \allowbreak\Z[\mathrm{x}]$, 
let $\bar{g}(\mathrm{x})\in \allowbreak\F_p[\mathrm{x}]$ be the
polynomial whose coefficients are those of $g(\mathrm{x})$,
reduced modulo $p$. We have 
\begin{equation}
\label{Modpf}
\bar{f}(\mathrm{x})=\allowbreak
 \bar{f_1}(\mathrm{x})^{e_1} \cdots \bar{f_r}(\mathrm{x})^{e_r}, \  
 e_i\geq 1, 
\end{equation} 
where $\bar{f_j}(\mathrm{x})$ is the mod-$p$ reduction of the 
polynomial $f_j(\mathrm{x}) \in 
\allowbreak \Z[\mathrm{x}]$ that has the same coefficients as 
$\bar{f_j}(\mathrm{x})$ but viewed as integers. 
The $\bar{f_j}(\mathrm{x})$ are $\F_p$-irreducible. Choose a root 
$\bar{\alpha_j}$ of each $\bar{f_j}(\mathrm{x})$, $1\leq j \leq r$. 
The map $\varphi_j$ on $\Z[\alpha]$ given by
\begin{equation}
\label{PhionZalpha}
\sum_{i=0}^{m-1}c_i\alpha^i \mapsto
 \sum_{i=0}^{m-1}\bar{c_i}   \bar{\alpha_j}^{\hspace{-1pt} i}
\end{equation}
is a well-defined ring epimorphism from $\Z[\alpha]$ onto 
$\F_p( \hspace{.5pt} \bar{\alpha_j}) = \F_{p^{k_j}}$, where $k_j$ 
is the degree of $f_j(\mathrm{x})$. If $p\nmid \mu$ then 
$\varphi_j$ extends to $\SL(n,\Z[\alpha, 1/\mu])$ and so is 
defined on $\SL(n,R_H)$.

For $S=\{a_1, \ldots , a_m\} \subseteq\P$, let $\Delta(S)$ be the determinant 
of the  $m \times m$ matrix with $(i,j)$th entry $\tr_{\P/\Q}(a_ia_j)$, where 
$\tr_{\P/\Q}(a)\in \Q$ is the trace of $a \in \P$ with respect to $\P/\Q$. The set 
$S$ is a $\Q$-basis of $\P$ if and only if $\Delta(S) \neq 0$~\cite[p.~342]{Koch}.
 Let $S$ be an integral basis of $\P$; i.e., all $a_i$ are in $\mathcal O$ and 
$S$ is a $\Z$-basis of $\mathcal O$. Then $\Delta(S)\in \Z\setminus\{0\}$.
 Also $\Delta(S') = \Delta(S)$ for any other integral basis $S'$ of $\P$. Thus 
$\Delta(\mathcal{O}):= \Delta(S)$ is the (uniquely defined) discriminant
 of $\P$~\cite[pp.~44--45]{Koch}. 

Let $p$ be a rational prime not dividing $\Delta(\alpha)\Delta(\mathcal{O})^{-1}$,
where $\Delta(\alpha) := \Delta(\{1, \alpha, \ldots , \allowbreak \alpha^{m-1}\})$.
 By \cite[Theorem~3.8.2]{Koch}, the ideal $p\mathcal{O}$ factorizes as 
$I_1^{e_1} \cdots I_r^{e_r}$ where $e_j$ and $r$ are as in \eqref{Modpf}, and the 
$I_j$ are distinct prime (so maximal) ideals of $\mathcal{O}$. Of course, this 
factorization is essentially unique; the $e_j$ and $I_j$ are determined by $p$.
Indeed $I_j$ is generated by $p$ and $f_j(\alpha)$. Thus, for each $p$ we get $r$ 
congruence homomorphisms $\varphi_{I_j}\colon \allowbreak \mathrm{SL}(n, \mathcal{O}) 
\rightarrow \mathrm{SL}(n, p^{k_j})$ such that $\varphi_{I_j}$
on $\SL(n,\Z[\alpha])$ is equal to $\varphi_j$ as defined by \eqref{PhionZalpha}. 

Each non-zero element $a$ of the Dedekind domain $R$ 
lies in just finitely many maximal ideals of $R$. 
The set of such ideals will be denoted by $\pi(a)$. 
Our algorithms return sets of maximal ideals of $R$, 
each ideal being given by at most two generators. When $R$ 
is a PID (principal ideal domain),
 each output ideal is designated by a single generator. 

\section{Excluding Aschbacher classes} 
\label{Asch}

Let $H \leq \SL(n, R)$ where $R= \mathcal{O}[1/\mu]$, 
and let $I$ be a maximal ideal of $R$ such that 
$R/I \cong \allowbreak \F_{p^k}$.  By default $H$ is given 
by a finite generating set. If $\varphi_I(H) \neq \SL(n, p^k)$
then $\varphi_I(H)$ is contained in a maximal subgroup of
$\SL(n, p^k)$ belonging to at least one of the Aschbacher 
classes $\calc_i$, $1 \leq \allowbreak i \leq \allowbreak 9$, 
for $\SL(n, p^k)$.
 We adopt the definitions of the $\mathscr{C}_i$ from \cite{Brayetal}
(where $\calc_9$ is denoted by $\cals$), and call $G\leq \SL(n,p^k)$ 
a \emph{$\calc_i$-group} if $G\leq K$ for a maximal subgroup $K$ of 
$\SL(n,p^k)$ in $\calc_i$. Since an irreducible but not absolutely 
irreducible subgroup of $\SL(n,p^k)$ is a $\calc_3$-group, we may 
assume that $\calc_2, \mathscr{C}_5, \calc_6, \calc_8, \calc_9$ consist 
of absolutely irreducible groups. 

In this section we define sets $\Pi_i(H)$ of maximal ideals of $R$,
such that if $\varphi_I(H)$ is a $\calc_i$-group for a maximal ideal 
$I$ of $R$, then $I\in \Pi_i(H)$. We give algorithms to compute the 
$\Pi_i(H)$. Similar algorithms for $H\leq \SL(n,\Z[1/\mu])$ appear 
in \cite{ExpMath}. These carry over for $\P \neq \Q$ with minor 
changes; however, the case $i=5$ is entirely new.

Unless stated otherwise, the degree $n$ in Subsections~\ref{Reducible}, 
\ref{Semilinear}, \ref{Class5}, and \ref{Isometry} is arbitrary.
Eventually we restrict to prime $n$, thereby simplifying 
the categorization of maximal subgroups of $\SL(n,p^k)$. 
For example, if $n$ is prime then $\calc_2$-groups are monomial, 
$\calc_3$-groups are metacyclic,  and the orders of 
$\calc_6$-groups and $\calc_9$-groups (not contained in any other
$\calc_i$) are bounded by a function of $n$, independent of $p$.
Additionally $\calc_4\cup \calc_7=\varnothing$ for prime $n$, so
 these classes are not mentioned henceforth.

\subsection{Reducible groups: $\calc_1$} 
\label{Reducible}

If $H\leq \GL(n,R)$ is absolutely irreducible, then $\varphi_I(H)$ 
is absolutely irreducible for almost all maximal ideals $I$ of $R$. 
We find the set $\Pi_1(H)$ of maximal ideals $I\subset R$ such that 
$\varphi_I(H)$ is not absolutely irreducible as follows 
(cf.~\cite[Section~2.2]{ExpMath}).

Let $\F/\mathbb{E}$ be a finite-degree field extension, and let 
$X$ be a finite subset of $\mathrm{Mat}(n,\F)$. The procedure 
${\tt BasisEnvAlgebra} (X, \mathbb{E})$ returns a basis $A$
of the enveloping algebra $\langle  X\rangle_{\mathbb{E}}$; this is 
the smallest subalgebra of the $\mathbb{E}$-algebra $\mathrm{Mat}(n,\F)$
 containing $X$. That is, $\langle X\rangle_{\mathbb{E}}$ comprises 
all $\mathbb{E}$-linear combinations of elements of the monoid generated 
by $X$. Denote the $\mathbb E$-linear span of $S\subseteq\mathrm{Mat}(n,\F)$
by $\mathrm{span}_{\mathbb{E}}(S)$. In detail, $\tt BasisEnvAlgebra$ 
initializes $A$ at $\{ 1_n\}$, and while there are $a \in A$ 
and $x \in X$ such that $ax \notin \mathrm{span}_{\mathbb{E}}(A)$, updates
$A$ to $A\cup \{ax\}$.

If $\F = \P$ and $H=\langle X\rangle$ where $X=X^{-1}$, then we 
write  ${\tt BasisEnvAlgebra}(H, \mathbb{E})$. In this case, the 
basis of $\langle X\rangle_{\mathbb E} =\langle H\rangle_{\mathbb E}$
returned is a subset of $H$. Further, if $\mathbb{E}=\F=\P$ then 
we shorten the procedure name to ${\tt BasisEnvAlgebra}(H)$.

To compute $\Pi_1(H)$, we first run ${\tt BasisEnvAlgebra}(H)$ to 
 get a $\P$-basis $A = \{a_1, \ldots, a_{d}\}$ of 
$\langle H \rangle_{\P}$. If  $d = n^2$ then $H$ is absolutely 
irreducible; otherwise, we have detected that $H$ is not
absolutely irreducible. Let $\Delta(A)$ be the determinant of the 
$n^2 \times n^2$ matrix with $(i,j)$th entry $\tr(a_ia_j)$, where
 $\tr$ denotes matrix trace. The determinant of a matrix 
whose columns are the $a_i$ written as column vectors 
squares to $\pm \hspace{.25pt} \Delta(A)$. So for any ideal 
$I \not \in \pi(\Delta(A))$, the $\varphi_I(a_i)$ are $R/I$-linearly
independent, and $\varphi_I(H)$ is absolutely irreducible. 
On the other hand, $\varphi_I(H)$ might be absolutely irreducible for 
$I \in \pi(\Delta(A))$. This can be checked for every ideal in the 
finite set $\pi(\Delta(A))$. Hence we have a procedure 
${\tt PrimesForAbsIrreducible}(H)$ that returns $\Pi_1(H)$ for 
absolutely irreducible $H$ (the `primes' being prime ideals, i.e., 
maximal ideals of $R$).

As noted in \cite[Section~2.2]{ExpMath}, if $\varphi_I(H)$ is 
absolutely irreducible for some ideal $I$ of $R$ then $H$ is absolutely 
irreducible, and we get a basis of $\langle H \rangle_{\P}$ by lifting 
up to $\GL(n,R)$ the elements of a basis of 
$\langle \varphi_{I}(H)\rangle_{R/I}$. 

\subsection{Imprimitive groups: $\calc_2$} 
\label{Imp}
Let $n$ be prime and $H\leq \SL(n, R)$ be irreducible 
(over $\P$). Then either $H$ is primitive or it is monomial---i.e., 
conjugate to a group of monomial matrices. The procedure 
${\tt PrimesForMonomial}(H)$ is defined as follows 
(cf.~\cite[Section~2.3]{ExpMath}). First  $g, h \in H$ are chosen 
such that $[g^e,h^e]\neq 1_n$, where $e$ is the exponent
$\mathrm{exp}(\mathrm{Sym}(n))$ of the symmetric group 
$\mathrm{Sym}(n)$. Let $c$ be the product of the non-zero entries 
of $[g^e,h^e]-1_n$.
The procedure tests primitivity of $\varphi_I(H)$ for all $I \in \pi(c)$ 
such that $\varphi_I(H)$ is absolutely irreducible, using the 
{\sf GAP} package \cite{NSrecog}. Since  $\varphi_I(H)$ is non-monomial 
for all $I \notin \pi(c)$, ${\tt PrimesForMonomial}(H)$ returns 
 the set $\Pi_2(H)$ of all maximal ideals $I$ of $R$ for which 
$\varphi_I(H)$ is absolutely irreducible and monomial. 

If $H$ is not solvable-by-finite, then elements $g,h$ as above exist
 in $H$ (by the Tits alternative) and are 
ubiquitous~\cite[Section~2.3]{ExpMath}. The implemented algorithm finds
$g, h$ by random selection.

\subsection{Semilinear groups: $\calc_3$} 
\label{Semilinear}

Let $H\leq \SL(n,R)$ be non-solvable and $l$ be a positive 
integer. We define a procedure ${\tt PrimesForSolvable}(H, l)$ 
to compute all maximal ideals $I\subset R$ such that  
$\varphi_I(H)$ is a proper solvable subgroup of $\SL(n,R/I)$ 
with derived length less than or equal to $l$ 
(cf.~\cite[Section~2.4]{ExpMath}).
Let $H = H^{(0)} \geq H^{(1)} = [H,H] \geq \cdots$ be the 
 derived series of $H$. The procedure randomly 
selects a non-trivial iterated commutator 
$g\in H^{(l+1)}$, and sets $c$ to be the product of all 
non-zero entries of $g - 1_n$. If $I \notin \pi(c)$ then
$\varphi_I(H)$ has derived length greater than $l$. With the 
aid of \cite{NSrecog}, ${\tt PrimesForSolvable}(H, l)$ returns 
the subset of $\pi(c)$ consisting of all $I$ such that 
$\varphi_I(H)\lneq \SL(n,R/I)$ has derived length at 
most $l$. 

When $n$ is prime, subgroups of $\SL(n, p^k)$ in the semilinear 
class $\calc_3$ are metacyclic. Hence for prime $n$ and 
non-solvable  $H\leq \SL(n,R)$, the set $\Pi_3(H):= 
{\tt PrimesForSolvable}(H,2)$ contains all maximal ideals $I$ 
of $R$ such that $\varphi_I(H)$ is a $\calc_3$-group in
 $\SL(n, R/I)$.

\subsection{Subfield groups: $\calc_5$} 
\label{Class5}

By \cite[Table~2.8]{Brayetal}, $\calc_5$ in $\SL(n, p^k)$
comprises conjugates of groups of the form
\[
\SL(n, p^r) . \big[
\mathrm{gcd}
\big({\textstyle \frac{p^k-1}{p^r-1}}, n \big) 
\big] \quad \text{where} \ \, k=rs,  \ s \ \text{prime}.
\]
Here $A.B$ denotes an extension of $A$ by $B$, and $[u]$ denotes
a group of order $u$. If $G\leq \SL(n,p^k)$ is a $\calc_5$-group, 
then $[G,G]\leq \SL(n,p^r)$ for some $r$ strictly dividing $k$.
We give conditions under which the  congruence image of an 
absolutely irreducible subgroup of $\SL(n, R)$ is a 
$\calc_5$-group. 

For a set $M$ of square matrices over a commutative ring 
$S$ with $1$, $\tr(M):= \{\tr(x) \mid \allowbreak x\in M\}$.
Note that $\tr(\SL(n, S)) = S$.
\begin{lemma} 
\label{Prec5}
Let $H \leq \SL(n, R)$ and $I$ be a maximal ideal of $R$ 
such that $R/I \cong \F_{p^k}$, $k > 1$. Suppose that $H$ 
is absolutely irreducible and $I\not \in \pi(\Delta(A))$ 
for some basis $A$ of $\langle H\rangle_\P$. If 
$\tr(H)\subseteq \widehat{R}$ where $\widehat{R}$ is a 
subring (with $1$) of $R$ such that 
$\varphi_I(\widehat{R}) \neq \F_{p^k}$, then $\varphi_I(H)$ 
is a $\calc_5$-group.
\end{lemma}
\begin{proof}
We have $\varphi_I(\widehat{R}) =\F_{p^r}$, $r<k$, and hence 
$\tr(\varphi_I(H)) = \varphi_I(\tr(H)) \subseteq \F_{p^r}$.
Also $\varphi_I(H)$ is absolutely irreducible because 
$\Delta(A) \notin I$. Then $\varphi_I(H)$ is conjugate to a 
subgroup of $\SL(n, p^r)$ by \cite[(9.14)]{IsaacsCT}.
\end{proof}

Now $\mathcal{O}[1/\mu]$
 is finitely generated as an additive group if and only if $\mu =1$,
so for the rest of the subsection we assume that $R=\mathcal O$. 
Let $\Pi_5(H)$ be the set of maximal 
ideals $I\subset \mathcal O$ such that $\varphi_I(H)$ is a 
$\calc_5$-group. We show how to decide whether $\Pi_5(H)$ is finite.

 Lemma~\ref{Prec5} suggests the need for a procedure to compute 
the subring $\mathrm{Tr}(H):= \langle \tr(H),1\rangle_\Z$ of 
$\mathcal O$ generated by $\tr(H)$ and $1$. In what follows we 
describe such a procedure.
 Note that the subring $\langle H \rangle_{\Z}$ of 
$\mathrm{Mat}(n,\mathcal{O})$ generated by $H$ coincides with
 $\mathrm{span}_\Z(H)$, the $\Z$-linear span of $H$. Certainly 
$\mathrm{Mat}(n,\mathcal{O})$ and so $\langle H \rangle_{\Z}$ 
have finite additive rank. Thus a $\Z$-spanning set $B\subseteq H$ 
of $\langle H \rangle_{\Z}$ can be found by a simple modification 
of ${\tt BasisEnvAlgebra}(H,\mathbb{E})$ from 
Section~\ref{Reducible}, replacing $\mathbb{E}$ by $\Z$.
The while-loop now tests whether $ax\in \mathrm{span}_\Z(A)$
 using an Hermite normal form (HNF) algorithm (see, e.g., 
\cite[Section~9.3]{Handbook}) rather than linear algebra over 
$\Q$. 
Since the trace map is additive,  
$\mathrm{span}_\Z(\tr(B)) = \allowbreak \mathrm{span}_\Z(\tr(H))$.
A $\Z$-basis of $\mathrm{Tr}(H)$ may then be computed in 
$\mathcal{O}=\allowbreak \mathrm{Mat}(1,\mathcal O)$ 
by running $\tt BasisEnvAlgebra$ over $\Z$ with input $\tr(B)$.

A $\Z$-submodule of $\mathcal O$ is \emph{complete} if its 
$\Z$-rank is $m=|\P:\Q|$. So a subring (with $1$) of $\mathcal O$ 
is complete if and only if it has field of fractions $\P$.
\begin{lemma} 
\label{FullRank}
Suppose that $\mathrm{Tr}([H,H])$ is a complete $\Z$-submodule 
of $\mathcal{O}$ with $\Z$-basis $B$, and let $w$ be a change 
of basis matrix from $B$ to a $\Z$-basis $A$ of $\mathcal O$.
Then $\varphi_I(H)$ is not a $\calc_5$-group for all maximal 
ideals $I$ of $\mathcal{O}$ such that 
$I \notin \pi(\mathrm{det}(w)^2\Delta(A))$.
\end{lemma}
\begin{proof}
By \cite[(2.4.1), p.~43]{Koch}, $\Delta(B) = 
\mathrm{det}(w)^2\Delta(A)$; so that if $I \notin 
\pi(\mathrm{det}(w)^2\Delta(A))$ then $I\not \in 
 \allowbreak \pi(\Delta(A) ) \cup \allowbreak \pi( \Delta(B))$.
Hence for any rational prime $p$ in $I$, 
$\varphi_I(\mathrm{Tr}([H,H]))$ and $\varphi_I(\mathcal{O})$ 
both have $\F_p$-dimension $m$; i.e., 
$\varphi_I(\mathrm{Tr}([H,H])) = \F_{p^m}$. Consequently 
$\varphi_I(H)$ is not a $\calc_5$-group: if it were, then we 
would have $\varphi_I(\mathrm{tr}([H,H]))\subseteq \F_{p^r}$, 
$r < m$. 
\end{proof}
\begin{corollary} 
\label{Pi5}
If $\mathrm{Tr}([H,H])$ is a complete $\Z$-submodule of 
$\mathcal{O}$ then $\Pi_5(H)$ is finite.
\end{corollary}

\begin{lemma} 
\label{NotFull} 
Let $H$ be absolutely irreducible. If $\Pi_5(H)$ is finite then
$\mathrm{Tr}(H)$ is complete. 
\end{lemma}
\begin{proof}
Suppose that  $\mathrm{Tr}(H)$ is not a complete $\Z$-module.
Let $A$ be a $\Z$-basis of $\mathcal O$.
 Then for any maximal ideal $I$ of $\mathcal O$ not in
$\pi(\Delta(A))$ and $p\in I$, the $\F_p$-dimension 
of $\varphi_I(\mathrm{Tr}(H))$ is less than $m$.
 Lemma~\ref{Prec5} implies that  $\Pi_5(H)$ is infinite.
\end{proof}

\begin{proposition}
\label{Pi5Infinite}
Suppose that $H$ is absolutely irreducible and $\tr(H)$ is 
contained in a proper subfield of $\P$. Then $\Pi_5(H)$ 
is infinite.
\end{proposition}
\begin{proof}
The field of fractions of $\mathrm{Tr}(H)$ is properly contained 
in $\P$. Hence $\mathrm{Tr}(H)$ is not complete. The result follows 
from Lemma~\ref{NotFull}.
\end{proof}

\begin{proposition}
\label{Pi5Finite}
Suppose that $[H,H]$ is absolutely irreducible and $\tr(H)$ is 
not contained in a proper subfield of $\P$. Then 
$\mathrm{Tr}([H,H])$ is complete and $\Pi_5(H)$ is finite.
\end{proposition}
\begin{proof}
We have $\mathrm{Mat}(n,\P)=\langle H \rangle_\Q =
 \langle [H,H] \rangle_\Q$. Since a basis of 
$\langle [H,H] \rangle_\Q$ may be chosen in $[H,H]$, 
the set $\tr([H,H])$ is not contained in a proper 
subfield of $\P$. Let $\F \subseteq \P$ be the field of 
fractions of $\mathrm{Tr}([H,H])$. If $\mathrm{Tr}([H,H])$ 
is not complete then $|\F : \Q| <m$, a contradiction.
By  Corollary~\ref{Pi5}, this gives the result.
\end{proof}

Since $\langle \mathrm{Tr}(H)\rangle_\Q= \langle B\rangle_\Q$ for any 
 $\Z$-spanning set $B$ of $\mathrm{Tr}(H)$, we see that $\tr(H)$ is 
contained in a proper subfield of $\P$ if and only if 
$\langle B\rangle_\Q$ has $\Q$-dimension less than $m$. So
$\tt BasisEnvAlgebra$ can be used to decide the criterion of 
Propositions~\ref{Pi5Infinite} and \ref{Pi5Finite}.

Assuming that $[H,H]$ is  absolutely irreducible, the preceding 
furnishes a procedure ${\tt Primes}$-${\tt ForSubfields}(H)$
 to decide whether $\Pi_5(H)$ is finite, returning $\Pi_5(H)$ 
in the event that this set is finite. We begin by testing
whether $\tr(H)$ lies in a proper subfield of $\P$. If it does, 
then we stop, reporting that  $\Pi_5(H)$ is infinite 
(Proposition~\ref{Pi5Infinite}). Otherwise, by Lemma~\ref{FullRank}
we compute $\Delta(B)$ for a $\Z$-basis $B$ of $\mathrm{Tr}([H,H])$.
Then $\Pi_5(H)$ is the set of $I \in \pi(\Delta(B))$ such that 
$\varphi_I(H)$ is a $\calc_5$-group; the latter checked by means of
\cite{NSrecog} as usual.

To find a $\Z$-basis $B$ of $\mathrm{Tr}([H,H])$, we first compute 
a $\Z$-basis $A\subseteq [H,H]$ of $\langle [H,H]\rangle_\Z$. 
This can be done using $\tt BasisEnvAlgebra$ as in the 
discussion before Lemma~\ref{FullRank}, and a modification of the 
procedure ${\tt BasisAlgebraClosure}(K,X)$ defined in 
\cite[Section~3]{TitsDFO}.
Here $X$ is a generating set of $H$ such that $X= X^{-1}$, and 
$K=[X,X]=\{ \hspace{.5pt} [x,y] \mid x,y\in X\}$. Note that $K$  
is a `normal generating set' for $[H,H]$, i.e., $[H,H]$ is the 
normal closure of $K$ in $H$. $\tt BasisAlgebraClosure$ 
initializes $A'$ to $K$, and its while-loop tests whether 
$a^x$ is in $\mathrm{span}_\Z(A')$ for all $a\in A'$ and $x\in X$, 
updating $A'$ to $A'\cup\{a^x\}$ if not.
Bases of $\Z$-modules are again computed via an HNF algorithm. The 
desired $A$ is any $\Z$-basis of $\langle A'\rangle_{\Z}$ 
(for $A'$ at completion of the while-loop). Then $B$ is found by 
another call on $\tt BasisEnvAlgebra$, with input $\tr(A)$.

\subsection{Bounded orders: $\calc_6$ and $\calc_9$} 
\label{class6}

Let $H$ be infinite, and $s$ be a positive integer. 
We proceed as in \cite[Section~2.1]{ExpMath} to determine all 
maximal ideals $I$ of $R$ such that $\varphi_I(H)$ has elements 
of order greater than $s$.
This leads to algorithms that are used to exclude $\calc_6$-groups 
and $\calc_9$-groups as congruence images.

Since it is finitely generated, $H$ must have an element  
of infinite order~\cite[4.9]{Wehrfritz}. In practice we often 
locate $h\in H$ of infinite order quickly just by random 
selection. For $1 \leq i \leq s$, let $m_i$ be the product of 
the non-zero entries of $h^i-1_n$, and set $l=m_1\cdots m_s$.
If $I$ is a maximal ideal of $R$ not in $\pi(l)$, then 
$\mathrm{exp}(\varphi_I(H))>s$. Using \cite{NSrecog} 
we can find all $I \in \pi(l)$ such that 
$\mathrm{exp}(\varphi_I(H)) \geq s$. Hence we get a procedure 
${\tt PrimesForOrder}(H, s)$ that accepts infinite $H\leq \SL(n, R)$
and a positive integer $s$, and returns the set of maximal ideals 
$I$ of $R$ for which $\mathrm{exp}(\varphi_I(H)) < s$.

We now consider groups in $\calc_6$. See \cite[Section~2.2.6]{Brayetal} 
for a definition of this class, which is nonempty only when $n$ is a 
power of a prime different from $p$.
\begin{lemma} 
\label{Cee6}
$\calc_6$-groups in $\SL(n, p^k)$ have orders bounded by a function of $n$ 
only, independent of $p$ and $k$.
\end{lemma}
\begin{proof}
This is clear from \cite[Table~2.9]{Brayetal}. 
\end{proof}

Let $\Pi_6(H)$ be the set of maximal ideals $I$ of $R$ such that 
$\varphi_I(H)$ is a $\calc_6$-group. By Lemma~\ref{Cee6},
 $\Pi_6(H)$ is finite. This set is found by running
${\tt PrimesForOrder}(H, s)$, where $s$ is a bound on the exponents
of all $\calc_6$-groups. Values for $s$ in small prime degrees  
(extracted from \cite[Chapter 8]{Brayetal}) are given in 
Section~\ref{StrongAT}.

We turn next to non-geometric $G\leq \SL(n,p^k)$.
By definition, $G\in \calc_9$ is almost simple modulo scalars, i.e., 
$K\unlhd G/Z\leq \mathrm{Aut}(K)$
for some non-abelian simple group $K$ such that $C_{G/Z}(K)$
is trivial, where $Z$ denotes the scalar subgroup of $G$. 
\begin{theorem}[{cf.~\cite[Lemma~3.1]{ExpMath}, \cite[Theorem~2.3]{SAT}}] 
\label{ZalesskiCee9}
Let $q$ be a prime and $\F=\F_{p^k}$. The orders of proper 
subgroups of $\hspace{.5pt} \SL(q,\F)$ contained in $\mathscr{C}_9$ and 
not in any other $\mathscr{C}_i$ are bounded above by a function of $q$ 
that does not vary with $p$.
\end{theorem}
\begin{proof}
Let $G$ be a maximal subgroup of $\SL(q,\F)$ in $\mathscr{C}_9$ such that 
$G\not \in \mathscr{C}_i$ for $i\neq 9$. Now $U= G^\infty$ 	(the intersection
 of all terms in the derived series of $G$) is a quasisimple normal subgroup 
of $G$. A bound on $|U|$ of the kind claimed implies the same
kind of bound on $|G|$, as $G/Z(G)$ embeds in $\mathrm{Aut}(U)$.  
Set $K= U/Z(U)$.

If $K$ is sporadic then of course $|K|$ is bounded absolutely.

If $K\cong \mathrm{Alt}(t)$ then $t\leq (3q+6)/2$ by 
\cite[Proposition~10, p.~333]{LubotzkySegal}, 
so $|K|$ is bounded by a function of $q$ only, independent of $\F$.

Suppose that $K$ is a group of Lie type $X_l(r^e)$ in cross characteristic 
$r\neq p$. Then \cite[Table~1]{SeitzZalesski} gives the smallest 
degree for which $K$ has a faithful projective representation over a 
field of characteristic $p$. These degree minima increase with 
the Lie rank $l$ and $r^e$, but are independent of $p$. Hence, for each $q$ 
and all $\F$, there are only finitely many possibilities for $|K|$.

This leaves the option that $K=X_l(p^e)$ is a group of Lie type in defining
 characteristic $p$. Existence of the required function for such $K$ is 
the main result of \cite{ZalesskiPreprint},
which crucially relies on the degree $q$ being prime---unlike 
the previous cases---and $G$ lying solely in $\mathscr{C}_9$. 
(Note also that we cannot use our proof of 
\cite[Lemma~3.1]{ExpMath} here, as $p$ could divide $|K|$, and so 
$K$ need not lift to characteristic zero.)
\end{proof}

Theorem~\ref{ZalesskiCee9} can be verified up to large values of $q$ using 
data available in \cite{Lubeck,Lubeck2}. Moreover, the theorem for $q\leq 11$
is also evident by inspection of the relevant tables in \cite{Brayetal}, and
this is more than enough for our experiments.

While the order bound of Theorem~\ref{ZalesskiCee9} is independent of $p$,  
it can depend on $k$ (see the example for $q=71$ in 
\cite[Section~7]{ZalesskiPreprint}). In our context, $k\leq |\P:\Q|$ will be 
fixed. We reiterate that the prime degree hypothesis is indispensable; e.g., 
 there are subgroups of $\SL(6,p^k)$ that are central extensions of $\SL(3,p^k)$, 
which belong to $\mathscr{C}_9$ and to no other 
$\mathscr{C}_i$~\cite[p.\hspace{1.5pt}389]{Brayetal}.

Let $\Pi_9(H)$ be the set of maximal ideals  $I \subset R$ 
such that $\varphi_I(H)$ is a $\calc_9$-group and not a $\calc_i$-group
for $i\neq 9$. Then $\Pi_9(H)$ is finite by Theorem~\ref{ZalesskiCee9}.
We compute $\Pi_9(H)$ using ${\tt Primes}$-${\tt ForOrder}(H, s)$, 
where $s$ bounds the exponents of $\calc_9$-groups not belonging to any 
other $\mathscr{C}_i$. Values of $s$ for $n\leq 11$ are 
stated in Section~\ref{StrongAT}.

\subsection{Isometry groups: $\calc_8$} 
\label{Isometry}

Let $\F$ be a field with an automorphism $\tau$ of order 
at most $2$. We will call $G\leq \GL(n, \F)$ an isometry group 
(over $\F$) if $G$ preserves a non-degenerate $\tau$-sesquilinear
form $\F^n\times \F^n\rightarrow \F$. So every subgroup of 
 $\SL(2,\F)=\Sp(2,\F)$ is an isometry group.
\begin{lemma}[{Cf.~\cite[Lemma~2.9]{ExpMath}}] 
\label{Unitary}
Let $G \leq \GL(n, \F)$.
\begin{itemize}
\item[{\rm (i)}] 
 If $G$ is an isometry group then $\tr(g) = \tr(\tau(g)^{-1})$
 for all $g \in G$.
\item[{\rm (ii)}]
If $G$ is absolutely irreducible and $\tr(g) = \tr(\tau(g)^{-1})$
 for all $g \in G$, then $G$ is an isometry group.
\end{itemize}
\end{lemma}
\begin{proof}
(i)~Let $\Omega$ be a matrix of a form preserved by $G$; i.e.,
$g^\top \Omega \tau(g) = \Omega$ for all $g\in G$. Non-degeneracy of 
the form implies that $\Omega$ is invertible. Thus
$\Omega^{-1}g^\top \Omega = \tau(g)^{-1}$, and (i) follows.

(ii)~In this case, the identity and contragredient representations 
of $G$ have equal characters. Since $G$ is absolutely irreducible,
these representations are equivalent: $g = \Omega^{-1}
(\tau(g)^\top)^{-1} \Omega$ for some $\Omega\in \allowbreak \GL(n,\F)$
and all $g\in G$. Hence $G$ preserves the form with matrix $\Omega$.
\end{proof}

Set $\tau = 1$. We define a procedure ${\tt PrimesForIsometry}(H)$
that accepts absolutely irreducible $H\leq \SL(n,R)$ and
searches for $h \in H$ such that $c := \tr(h) - \tr(h^{-1}) \neq 0$. 
If this search succeeds, then using \cite{NSrecog} the procedure
returns the set $\Pi_8(H)$ of maximal ideals $I$ of $R$ such that 
$I \in \pi(c)$ and  $\varphi_I(H)$ preserves an alternating or 
symmetric bilinear form over $R/I$.
 Suppose that $H$ does not preserve an alternating or symmetric 
form; so Lemma~\ref{Unitary}~(ii) guarantees that $h$ as required 
exists. By Lemma~\ref{Unitary}~(i), $\varphi_I(H)$ for any maximal 
ideal $I\subset R$ preserves an alternating or symmetric form if and 
only if $I \in \Pi_8(H)$.

Recall that $g\in \GL(n,\F)$ is a \emph{similarity} of a
sesquilinear form $f\colon \F^n\times \F^n\rightarrow \F$
if there exists $a\in \allowbreak \F^\times$ such 
that $f(gu,gv) = a f(u,v)$ for all $u, v\in \F^n$. We will 
need to exclude congruence images of $H$ that are in a 
similarity group. To achieve this, for input $H$ such that 
$[H,H]$ is absolutely irreducible, we modify the initial step 
of ${\tt PrimesForIsometry}(H)$: now we search for $h\in [H,H]$ 
such that $\tr(h) - \tr(h^{-1}) \neq 0$ (e.g., commutators 
$[h_1,h_2]$, $h_i\in H$ could be tried as candidates for $h$).
 The final step is to check whether $\varphi_I(H)$ 
lies in a similarity group of a symmetric or alternating bilinear 
form. We name the resultant procedure
 ${\tt PrimesForSimilarity}(H)$. 
\begin{lemma}[{Cf.~\cite[Lemma~2.10]{ExpMath}}] 
\label{iso1}
Let $H\leq \SL(n,R)$ and suppose that $[H,H]$ is absolutely 
irreducible and does not preserve an alternating or symmetric 
bilinear form over $\P$. Then $\varphi_I(H)$ is not 
contained in a similarity group of an alternating or symmetric 
form over $R/I$ for almost all maximal 
ideals $I$ of $R$.
\end{lemma}
\begin{proof}
This follows from Lemma~\ref{Unitary} and the definition of
 $\tt PrimesForSimilarity$.
\end{proof}

The class $\calc_8$ in $\SL(n,p^k)$ comprises similarity 
groups of non-degenerate classical forms over $\F_{p^k}$.
We split $\calc_8$ into the union of two subsets, $\calc_8^{(1)}$ 
(orthogonal and symplectic) and $\calc_8^{(2)}$ (unitary). 
Likewise, we re-define $\Pi_8 (H)= \Pi_8^{(1)}(H) \cup \Pi_8^{(2)}(H)$,
where $\Pi_8^{(1)}(H)= {\tt PrimesFor}$-${\tt Similarity}(H)$ and 
$\Pi_8^{(2)}(H)$ is the set of all maximal ideals $I$ of $R$ such that  
$\varphi_I(H)$ is contained in a similarity group of an Hermitian form 
over $R/I$.

\subsection{Congruence images in prime degree}
\label{Summation}

We recap the procedures of this section. Let $n$ be a prime, and 
let $H \leq \SL(n, R)$ be absolutely irreducible, primitive, 
non-solvable, and infinite.

If $i\not \in \{5,8\}$ then $\varphi_I(H)$ is not a 
$\calc_i$-group for almost all maximal ideals $I$ of $R$. 
 In these cases ${\tt PrimesForAbsIrreducible}$,
 ${\tt PrimesForMonomial}$, 
 ${\tt PrimesForSolvable}$, and ${\tt PrimesForOrder}$
compute the finite sets $\Pi_i(H)$ of maximal ideals 
$I\subset R$ such $\varphi_I(H)$ is a $\calc_i$-group. 

Suppose that $[H,H]$ is absolutely irreducible.
 ${\tt PrimesForSimilarity}$ returns the set $\Pi_8^{(1)}(H)$
of maximal ideals $I$ such that 
$\varphi_I(H)$ is a $\calc_8^{(1)}$-group.
If $R= \mathcal{O}$ then ${\tt PrimesForSubfields}$
decides finiteness of
the set $\Pi_5(H)$ of maximal ideals $I$ such that
$\varphi_I(H)$ is a $\calc_5$-group,
returning $\Pi_5(H)$ when it is finite.

\section{Computing congruence quotients of dense subgroups} 
\label{StrongAT}

In this section we give an algorithm to find the 
 congruence images of a finitely generated dense group 
$H\leq \SL(n, R)$ modulo all maximal ideals of $R=
\mathcal{O}[1/\mu]$. 
 This will leave only a finite set
 of ideals $\Pi(H)$ for which the congruence images
 should be computed.
In particular, we formulate conditions under which $H$ 
 surjects onto $\SL(n, R/I)$ when $I \notin \Pi(H)$.
We start by establishing various 
consequences of density 
that provide more background for the algorithms.

Recall that the adjoint module for $\SL(n, \F)$ is the 
$\F$-vector space 
\[
\mathfrak{sl}(n, \F) = 
\{x \in \mathrm{Mat}(n, \F) \, | \,
 \tr(x) = 0\}
\]
of dimension $n^2-1$ on which $\SL(n, \F)$ acts by conjugation.
This action gives rise to the  adjoint representation 
$\Ad \colon  \SL(n, \F) \rightarrow \GL(n^2-1, \F)$.
\begin{proposition}[{\cite[9.1]{Rivin}}]
\label{Rivin91}
A subgroup $G$ of $\SL(n,\C)$ is dense if and only 
if $G$ is infinite and $\Ad(G)$ is absolutely irreducible.
\end{proposition}

\begin{proposition} 
\label{AdjModule}
Let $G\leq \SL(n,\C)$ be dense. Then $G$ is primitive (irreducible) over $\C$, 
non-solvable, and does not preserve a non-degenerate alternating 
or symmetric bilinear form over $\C$ if $n>2$.
\end{proposition}
\begin{proof}
(Cf.~the proof of \cite[Proposition~2.1]{SAT}.)
By Proposition~\ref{Rivin91}, $G$ is infinite and 
$\Ad(G)\leq \GL(n^2-1,\C)$ is irreducible.

Suppose that $G$ is reducible.
Then $G$ is conjugate to a block upper triangular group with main diagonal 
$(G_1, G_2)$ where $G_i$ has degree $n_i<n$. It is easily seen that $\Ad(G)$ 
leaves invariant the proper non-zero subspace of $\mathfrak{sl}(n, \C)$ 
consisting of all block upper triangular matrices with main diagonal 
$(0_{n_1},0_{n_2})$.
This contradiction shows that $G$ is absolutely irreducible.

If $G$ is imprimitive over $\C$, then $G\leq K:= 
\allowbreak \GL(a,\C) \hspace{1pt} \wr \hspace{1pt} 
\mathrm{Sym}(b)$ up to conjugacy, for some 
$a, \allowbreak b$ such that $b>1$ and $ab=n$.
Let $U$ be the subspace of $\mathfrak{sl}(n,\C)$ 
spanned by all block diagonal matrices with zero trace 
whose blocks are in $\{\pm 1_a,0_a\}$. Clearly $U$ has 
dimension $b-1$ and is invariant under conjugation by
 $K$. Hence $G$ must be primitive.

By \cite[3.4]{Wehrfritz} (a result of Zassenhaus), a 
solvable primitive subgroup of $\GL(n,\C)$ is 
central-by-finite. Thus, because scalar subgroups of
$\SL(n,\C)$ are finite, $G$ cannot be solvable.

Finally, suppose that $G$  preserves a non-degenerate 
symmetric or alternating bilinear form over $\C$, with
Gramian $\Phi$. Then 
$\Ad(G)$ leaves invariant the adjoint module 
$L= \allowbreak
\{ x\in \mathfrak{sl}(n,\C) \, | \, \allowbreak
x^\top\Phi +\Phi x = 0_n\}$.
 However, $L$ is a proper non-zero subspace of 
$\mathfrak{sl}(n,\C)$ if $n>2$. 
\end{proof}

Parts of Proposition~\ref{AdjModule} also follow from basic 
facts about the Zariski topology on $\GL(n,\C)$.
\begin{lemma}
\label{DerivDense}
Let  $G\leq \SL(n,\C)$. Then $G$ is dense in $\SL(n,\C)$ 
if and only if $[G,G]$ is dense.
\end{lemma}
\begin{proof}
Denote the Zariski closure of $K\leq \SL(n,\C)$ in $\SL(n,\C)$ 
by $\overline{K}$. It is not difficult  to prove that 
$[\hspace{1.25pt} \overline{G}, \overline{G}\hspace{1.25pt} ] 
\subseteq \overline{[G,G]}$; see e.g., \cite[5.10]{Wehrfritz}.
\end{proof}

 Strong approximation implies Lemma~\ref{DerivDense}
when $G\leq \SL(n,\Q)$ and $G$ is finitely generated. 

\begin{lemma}
\label{PrimSFisSolvable}
Let $q$ be a prime. If $H\leq \SL(q,R)$  is infinite, 
solvable-by-finite, and primitive over $\C$, then $H$ is 
solvable.
\end{lemma}
\begin{proof}
The proof of \cite[Lemma~2.4]{ExpMath} is valid for subgroups 
of $\SL(q, R)$. 
\end{proof}

\begin{remark}
\label{DenseQuick}
Since the Zariski closure of a solvable subgroup of $\SL(n,\C)$ 
is a solvable subgroup of $\SL(n,\C)$~\cite[5.11]{Wehrfritz},
dense $G\leq \SL(n,\C)$ is not solvable-by-finite.
\end{remark}

For each maximal ideal $I$ of $R$, let $\mathcal{X}_n(R/I)$ be the 
union of $\{\SL(n, R/I)\}$ and the set of all proper subgroups $G$ 
of $\SL(n,R/I)$ such that $G$ is a $\calc_5$-group or 
$\calc_8^{(2)}$-group, and not a $\calc_8^{(1)}$-group nor a 
$\calc_i$-group for $i=1,2,3,6$.
We denote by $\Pi(H)$ the set of maximal ideals $I$ of $R$ such that 
$\varphi_I(H)\not \in \mathcal{X}_n(R/I)$.
\begin{theorem}
\label{SATN}
Let $q$ be a prime and $H$ be a finitely generated subgroup of 
$\SL(q,R)$ dense in $\SL(q,\C)$. Then $\Pi(H)$ is finite.
\end{theorem}
\begin{proof}
Certainly $\Pi(H)$ is contained in the union of $\Pi_8^{(1)}(H)$ 
and the $\Pi_i(H)$, $i = 1,2,3,6,9$. We apply Proposition~\ref{AdjModule} 
throughout. Subsections~\ref{Reducible}--\ref{Semilinear} and 
Lemma~\ref{PrimSFisSolvable} (or Remark~\ref{DenseQuick})
imply that the $\Pi_i(H)$ for $i = 1,2,3$, are finite. 
By Lemma~\ref{Cee6} and Theorem~\ref{ZalesskiCee9}, $\Pi_6(H)\cup \Pi_9(H)$ is 
finite. Also $\Pi_8^{(1)}(H)$ is finite by Lemmas~\ref{iso1} 
and \ref{DerivDense}.
\end{proof}

Theorem~\ref{SATN} and its proof give us a procedure 
${\tt PrimesForDense}(H)$ that computes $\Pi(H)$ for finitely 
generated dense $H \leq \SL(q, R)$. It combines outputs of the 
following procedures from Section~\ref{Asch} (summarized in 
Subsection~\ref{Summation}):

\vspace{7.5pt}

\begin{itemize}
\item ${\tt PrimesForAbsIrreducible}(H)$;

\vspace{3pt}

\item ${\tt PrimesForMonomial}(H)$;

\vspace{3pt}

\item ${\tt PrimesForSolvable}(H, 2)$;

\vspace{3pt}

\item ${\tt PrimesForSimilarity}(H)$, $n>2$ only;

\vspace{3pt}

\item ${\tt PrimesForOrder}(H, s)$, where $s$ is a bound on the
 exponents of $\calc_6$-groups and of $\calc_9$-groups belonging to no 
other $\calc_i$.
\end{itemize}

\vspace{7.5pt}

\noindent All $I$ from the above output such that
 $\varphi_I(H)\in \mathcal{X}_n(R/I)$ are 
deleted, and the procedure returns the surviving ideals of $R$. 
Note that if $q = 2,3,\allowbreak  5,\allowbreak 7,11$ then we 
can take $s$ to be $10, 21, 60, 84, 253$, respectively, by 
\cite[Remark~3.3]{ExpMath}; these values were obtained using {\sf GAP}
 and tables in \cite[Chapter~8]{Brayetal}.

As we know, in contrast to strong approximation for finitely generated
dense subgroups of $\SL(n,\Q)$, if $R \not \subseteq \Q$ 
then there could be infinitely many maximal ideals $I$
of $R$ such that $\varphi_I(H)\neq \SL(n,R/I)$. In other
words, $\Pi_5(H)$ or $\Pi_8^{(2)}(H)$ could be infinite (examples 
of the latter eventuality include certain 
$H\leq \SL(n,\Z[\hspace{1pt}\mathrm{i}\hspace{1pt}])$ 
preserving an Hermitian form; cf.~\cite[Section~5]{Bridson}).
However, by Lemma~\ref{DerivDense}, Subsection~\ref{Class5}, and
the above,
at least when $R=\mathcal{O}$ and $\Pi_8^{(2)}=\varnothing$, and 
we have found that $\Pi_5(H)$ is finite, we can compute the full 
set of maximal ideals $I$ such that $\varphi_I(H) \neq \SL(q,R/I)$.

\section{Experiments} 
\label{Exper}

Let $R=\mathcal{O}[1/\mu]$ and $q$ be a prime.
 We investigate a given finitely generated dense subgroup 
$H$ of $\Gamma := \allowbreak \SL(q, R)$ by running 
${\tt PrimesForDense}(H)$ to compute $\Pi(H)$. Then we find 
$\varphi_I(H)$ for each $I \in \Pi(H)$. Since $\varphi_I(H)\neq
\SL(n, R/I)$ might be a $\calc_i$-group for several $i$, we do 
not say what these classes are, but rather state the isomorphism type 
of $\varphi_I(H)$.  

We can learn more about $H$. For example, sometimes we are
able to decide whether  $H$ is thin, or whether it is a congruence 
subgroup. Here, the dense group $H$ is \emph{thin} if $|\G : H|$ 
is infinite. Also, a (finite-index) subgroup of $\G$ is a 
\emph{congruence subgroup} if it contains a 
\emph{principal congruence subgroup} (PCS), i.e., the kernel of 
$\varphi_I$ on $\G$ for some ideal $I$ of $R$. We call  
this ideal the \emph{level} of the PCS.

Experiments were conducted using our {\sf GAP}~\cite{Gap} 
implementation of the algorithms. We restrict to $q=2$
only, which abbreviates the list of subprocedures involved
 (practicality of the algorithms is not substantially 
affected by increasing the degree; cf.~the algorithms of \cite{ExpMath},
 over $\Q$). 
For if $q=2$ then groups in $\calc_1\cup \calc_2$ are 
solvable of derived length at most $4$, and $\calc_8= \allowbreak 
\varnothing$ (cf.~\cite[Remarks~2.8, 3.4]{ExpMath}).
Thus ${\tt PrimesForDense}(H)$ in this section combines 
${\tt PrimesFor}$-${\tt Solvable}(H, 4)$ and ${\tt PrimesForOrder}(H, 10)$
 only. 
Observe too that $H\leq \SL(2, \C)$ is dense if and only if it is 
not solvable-by-finite. Hence the main algorithm of \cite{TitsDFO} tests 
density of $H$. 

Machinery to compute with groups over number fields $\P$ was developed 
along the lines of Section~\ref{CongHom}. For experiments 
in this paper and \cite{SLnO}, 
$\P$ is contained in a cyclotomic field, and current 
{\sf GAP} functionality allows us to perform the necessary algebraic number 
theory calculations. No bottlenecks were encountered.
To determine $\pi(a)$ given $a\in \mathcal{O}$, we find the norm  
$\mathrm{N}_{\P/\mathbb{Q}}(a) \in \allowbreak \Z$, then factorize it over $\Z$ 
as a product of powers of primes $p$. Factorizing each $p$ over 
$\mathcal{O}$ is standard (Section~\ref{CongHom}), and for quadratic 
extensions this is simpler still~\cite[Proposition~3.8.3]{Koch}.
 	
We show how our algorithms can be used to answer questions about the
 subgroup structure of classical groups over rings. Actually all input groups 
are subgroups of $\SL(2,\mathcal{O}_d)$, where $\mathcal{O}_d$ denotes the 
ring of integers of $\P=\Q(\sqrt{-d}\hspace{.75pt})$ for some square-free 
positive $d\in \Z$. The projective group $\PSL(2,\mathcal{O}_d)$ is
 a \emph{Bianchi group}~\cite{Fine}.

The {\sf GAP} code used in the experiments
may be accessed at \url{https://github.com/hulpke/arithmetic}.

\subsection{Thinness testing}
\label{O7}
We start with a demonstration of the main algorithm. Then we test thinness 
of subgroups of $\SL(n, R)$ when the congruence subgroup property (CSP) does 
not hold; i.e., $\SL(n,R)$ has finite-index subgroups that are not
 congruence subgroups. By results of Serre~\cite{Serre}, the CSP does
not hold in $\SL(2,\mathcal{O}_d)$.

Let $\G = \SL(2, \mathcal{O}_7)$ and
\[
H = \big\langle \begin{bmatrix}
                 1&0\\a&1\end{bmatrix},
                \begin{bmatrix} 1&b 
								\\ 0&1 \end{bmatrix} 
	\big\rangle \quad \text{where} \ \, a = \frac{37-\sqrt{-7}}{2}, \, 
	 b = 2-\sqrt{-7}.
\]
We have $R_H=\mathcal{O}_7=\Z[\frac{1+\sqrt{-7}}{2}]$. 
The ring  $\mathrm{Tr}(H)$ is 
generated by $\big\{2, \frac{71-39\sqrt{-7}}{2}\big\}$, and
 $\mathrm{Tr}(H)$ has index $78=2\cdot 3\cdot 13$ in
 $\mathcal{O}_7$. Since $\tr(H) \not\subseteq \allowbreak 
\Q$, $\Pi_5(H)$ is finite.
We run ${\tt PrimesForDense}(H)$ and ${\tt PrimesForSubfields}(H)$ 
 to compute $\Pi(H)$, now defined to be the set of all maximal 
ideals $I\subset \mathcal{O}_7$ such that $\varphi_I(H)
\neq \SL(2,\mathcal{O}_7/I)$. Then for each $I \in \Pi(H)$ we 
find $\varphi_I(H)$. Results are collected in Table~\ref{tab:table1}.

\begin{table}[h!]
\begin{center}
\caption{Congruence images of 
$H\le\SL(2, \mathcal{O}_7)$}
\label{tab:table1}
\begin{tabular}{|c|c|c|c|}
\hline
Prime & Norm & Index  & Structure \\
\hline
\vspace{-12pt} & & &  \\
$\frac{-1+\sqrt{-7}}{2}$&2&3&$C_2$\\
\vspace{-12pt} & & &  \\
$3$&9&30&$\SL(2,3)$\\
\vspace{-12pt} & & &  \\
$-2\sqrt{-7}$&11&120&$C_{11}$\\
\vspace{-12pt} & & &  \\
$13$&169&2210&$\SL(2,13)$\\
\vspace{-12pt} & & &  \\
$6+\sqrt{-7}$&43&1848&$C_{43}$\\

\hline
\end{tabular}
\end{center}
\end{table}

 All $I\in \Pi(H)$ are listed in the column `Prime' 
(each given by a single ideal generator,
 because $\mathcal{O}_7$ is a PID), `Norm' is the size of 
$\mathcal{O}_7/I$, `Index' is 
$|\SL(2, \mathcal{O}_7/I):\varphi_I(H)|$, and `Structure' 
records the isomorphism type of $\varphi_I(H)$ ($C_t$ is
 the cyclic group of order $t$). If $|\G:H|$ were finite,  
then it would be divisible by the least common multiple of all 
entries in the `Index' column.
 
In fact, as we now show, $H$ is thin (cf.~\cite[pp.~84--85]{Hulpke}).
Let $N\leq \G$ be the PCS of level $3\mathcal{O}_7$. Then 
$\G/N \cong \SL(2,9)$, and working from a presentation of $\G$ 
(see \cite[Theorem~13.1]{Swan}), we find that the abelianization 
$N/[N,N]$ of $N$ is free abelian of rank $40$. However, the image of 
$H\cap N$ under natural surjection $N\rightarrow N/[N,N]$ has rank 
$11$. Thus $|N:(H\cap N)| = |HN:H|$ is infinite, and so the index of
$H$ in $\G$ is infinite.

\subsection{Identifying congruence subgroups}

Let $\G = \SL(2, \mathcal{O}_3)$. Note that $\mathcal{O}_3$ 
is a PID, and $\mathcal{O}_3 = \allowbreak \Z[\omega]$
 where $\omega$ is the primitive cube root
of unity $\frac{-1+\sqrt{-3}}{2}$. In this subsection we 
 identify congruence subgroups among the finite-index 
 subgroups of $\G$.

The group $\PSL(2, \mathcal{O}_3)$ differs essentially from 
other Bianchi groups. Indeed, $\PSL(2, \mathcal{O}_3)$
has Serre's property FA, hence is not a free product with 
amalgamation nor an HNN extension~\cite[Corollary 1]{Alperin}, 
and its abelianization is finite~\cite[p.~2939]{Alperin}.
For $d \neq 3$, $\PSL(2, \mathcal{O}_d)$ is a free product with 
amalgamation~\cite[Theorem~6.3.1, p.~162]{Fine}. Furthermore, 
$\PSL(2, \mathcal{O}_3)$ is profinitely rigid; the proof relies 
on strong approximation~\cite[Section 6.2]{ReidAnnMath}.

By \cite[Theorem~6.1]{Swan}, $\Gamma$ has presentation
\[
\langle \hspace{.75pt} t,l,a \, | \, l^3,a^4,
 lala^{-1},a^2ta^{-2}t^{-1},
(t^{-1}a)^3, (tl^{-1})^3, (tl)^3  \hspace{.75pt} \rangle
\]
(with redundant generators eliminated), where
\[
t=\begin{bmatrix}1&1\\0&1\end{bmatrix},\quad
l=\begin{bmatrix}\omega^2&0\\0&\omega
\end{bmatrix},\quad
a=\begin{bmatrix}0&-1\\1&\ \ 0\end{bmatrix}.
\]
\noindent 
We test whether finite-index subgroups of $\G$ are congruence 
 subgroups. A low-index subgroup computation (see 
\cite[Section~5.4]{Handbook}) determines all $126$ proper 
subgroups of $\G$ up to conjugacy that have index at most $30$. 
Eight of these, labeled $L_i$ for $1\leq i \leq 8$ in 
Table~\ref{tab:table2}, are not congruence subgroups. This is 
because the image of $\G$ under the permutation representation 
arising from the action by $\G$ on cosets of each $L_i$ is an 
alternating group $\mathrm{A}_r$ for $r \ge 22$. Any PCS in $L_i$ 
is contained in the kernel of the coset action. 
By the Chinese remainder theorem, we may assume 
that $L_i$ contains a PCS of prime-power level, say $I^j$ 
where $I\subset \mathcal{O}_3$ is maximal. Since the kernel
of reduction modulo $I$ on $\SL(2,\mathcal{O}_3/I^j)$ is 
solvable, $\mathrm{A}_r$ is then a quotient of some 
$\SL(2,p^k)$. But this is impossible (see, e.g., 
\cite[Proposition~10, p.\hspace{1.5pt}333]{LubotzkySegal}).

\begin{table}[h!]
 \begin{center}
 \caption{Non-congruence subgroups of $\SL(2, \mathcal{O}_3)$}
 \label{tab:table2}
 \begin{tabular}{|c|c|c|} \hline
$i$ & $L_i$ & $|\G : L_i|$ \\
 \hline
\vspace{-12pt} & &   \\
1&$\langle \hspace{.75pt} l, a, tl^{-1}t^{-2}, t^{-2}at^{2}a^{-1}t^{-2} \hspace{.75pt} \rangle$& $22$\\
\vspace{-12pt} & &   \\
2&$\langle \hspace{.75pt} l, a, tlt^{-2}, t^{-2}at^{2}a^{-1}t^{-2}  \hspace{.75pt} \rangle$&$22$\\
\vspace{-12pt} & &   \\
3&$\langle \hspace{.75pt} t^{3}, tat^{-1}, lt^{3}l^{-1}t, ltat^{-1}l^{-1} \hspace{.75pt} \rangle$&$27$\\
\vspace{-12pt} & &   \\
4&$\langle \hspace{.75pt} a, t^{3}, lt^{3}l^{-1}t^{-1}, ltat^{-1}l^{-1}  \hspace{.75pt} \rangle$&$27$\\
\vspace{-12pt} & &   \\
5&$\langle \hspace{.75pt} t^{-2}, l, a^{-2}, at^{3}a^{-1}, tlatlt^{-2}a^{-1}t^{-1} \hspace{.75pt} \rangle$&$28$\\
\vspace{-12pt} & &   \\
6&$\langle \hspace{.75pt} l, a, t^{3}, tlt^{-1}atl^{-1}t^{-1}, tat^{-3}a^{-1}t  \hspace{.75pt} \rangle$&$29$\\
\vspace{-12pt} & &   \\
7&$\langle \hspace{.75pt} a, lta^{-1}t^{-1}, l^{-1}t^{-2}l  \hspace{.75pt} \rangle$&$30$\\
\vspace{-12pt} & &   \\
8&$\langle \hspace{.75pt} a, lt^{-2}l^{-1}, l^{-1}ta^{-1}t^{-1}  \hspace{.75pt} \rangle$&$30$\\
                \hline
    \end{tabular}
  \end{center}
\end{table}
\noindent For the groups $H = L_i$ in Table~\ref{tab:table2}
 we have $R_H=\mathrm{Tr}(H)=\mathcal{O}_3$, and after running
${\tt PrimesForDense}$ we find that they surject onto 
$\SL(2, \mathcal{O}_3/I)$ modulo every maximal ideal $I$. 

A further five subgroups $L_i$, as in Table~\ref{tab:table3}, 
surject modulo every maximal ideal $I$ of $\mathcal{O}_3$
onto $\SL(2,\mathcal{O}_3/I)$. 

\begin{table}[h!]
  \begin{center}
    \caption{Congruence subgroups of $\SL(2, \mathcal{O}_3)$}
    \label{tab:table3}
    \begin{tabular}{|c|c|c|}
      \hline
      $i$ & $L_i$ & $|\G : L_i|$ \\
      \hline
			\vspace{-12pt} & &  \\
9&$\langle \hspace{.75pt} l, talt^{-1}, t^{2}lta^{-1}t^{-2} 
 \hspace{.75pt} \rangle$&$16$\\
\vspace{-12pt} & &  \\
10&$\langle \hspace{.75pt} l, tal^{-1}t^{-1}, t^{2}lt^{-1}a^{-1}t^{-2}  
 \hspace{.75pt} \rangle$&$16$\\
\vspace{-12pt} & &  \\
11&$\langle \hspace{.75pt} l, a, tat^{-2}, t^{-4}  
 \hspace{.75pt} \rangle$&$16$\\
\vspace{-12pt} & &  \\
12&$\langle \hspace{.75pt} l, a, tat^{-1}, t^{-4}  
 \hspace{.75pt} \rangle$&$16$\\
\vspace{-12pt} & &  \\
13&$\langle \hspace{.75pt} t,a,lt^{3}l^{-1},l^{-1}tat^{-2}a^{-1}tl^{-1} 
 \hspace{.75pt} \rangle$&$27$\\
                \hline
    \end{tabular}
  \end{center}
\end{table}

If $H$ is any one of $L_9,\ldots , L_{12}$ then 
$R_H= \mathrm{Tr}(H)=\mathcal{O}_3$, while 
$\mathrm{span}_\Z(\tr(L_{13}))= \allowbreak 3\mathcal{O}_3$. 
However, all groups in Table~\ref{tab:table3} are congruence 
subgroups: we verified that $|\varphi_{24}(\Gamma):\varphi_{24}(L_i)|
=|\Gamma:L_i|$ for these $i$, where $\varphi_{24}$ is the congruence
homomorphism modulo $24$, so $\ker \varphi_{24}\leq L_i$.
This does not contradict 
the fact that each $L_i$ surjects onto $\SL(2, \mathcal{O}_3/I)$ 
modulo every maximal ideal $I\subset \mathcal{O}_3$. 
Indeed, for small $n$ and small $R/I$, we know that $\SL(n,R/I)$ has 
supplements to images of principal congruence subgroups; 
cf.~\cite[Section~2.3]{Density}.

All other subgroups (of the $126$) do not surject onto 
$\SL(2,\mathcal{O}_3/I)$ modulo at least one  maximal  ideal $I$. 
Of these, $20$ are not congruence subgroups, as the image of $\G$ 
arising from its action on their cosets does not occur in any 
$\SL(2,p^k)$. These non-congruence subgroups are labeled 
$L_{14},\ldots, L_{33}$ in Table~\ref{Tabe4}. 

\begin{table}[h!]
  \begin{center}
    \caption{More non-congruence subgroups of $\SL(2, \mathcal{O}_3)$}
    \label{Tabe4}
    \begin{tabular}{|c|c|c|c|c|}
      \hline
$i$ & $L_i$ & $|\G : L_i|$&Reason&Prime \\
\hline
\vspace{-12pt} & & & &   \\
14& $\langle \hspace{.75pt} t^{-4}, t^{-1}la^{-1}t^{-1}, lat^{-2} 
\hspace{.75pt}\rangle$&24&RA& $2$ \\
\vspace{-12pt} & & & &   \\
15& $\langle \hspace{.75pt} ta^2, at^2l^{-1} \hspace{.75pt}\rangle$
&24&RA& $2$ \\
\vspace{-12pt} & & & &   \\
16& $\langle \hspace{.75pt} t, at^{-2}l^{-1} \hspace{.75pt}\rangle$
&24& RA& $2$ \\
\vspace{-12pt} & & & &   \\
17& $\langle \hspace{.75pt} a, tal^{-1}t^{-1}, 
 t^{-1}al^{-1}t, lt^{-4}l^{-1} \hspace{.75pt}\rangle$
&24& RA& $2$ \\
\vspace{-12pt} & & & &   \\
18& $\langle \hspace{.75pt} l, a, t^3, tlat^{-2}a^{-1}t \hspace{.75pt}\rangle$
&24& RB& $\frac{-3+\sqrt{-3}}{2}$ \\
\vspace{-12pt} & & & &   \\
19& $\langle \hspace{.75pt} l, a, tl^{-1}atl^{-1}t^{-1} \hspace{.75pt}\rangle$
&28& RC& $\frac{5+\sqrt{-3}}{2}$ \\
\vspace{-12pt} & & & &   \\
20& $\langle \hspace{.75pt} l, a, tl^{-1}t^{-1}a^{-1}l^{-1}t^{-1}\rangle$ &28& RC&
$\frac{-5+\sqrt{-3}}{2}$ \\
\vspace{-12pt} & & & &   \\
21& $\langle \hspace{.75pt} l, a, talt^{-1}, t^2al^{-1}t^{-2}, t^4lt^{-2}\rangle$ &28& RC&
$\frac{-5+\sqrt{-3}}{2}$ \\
\vspace{-12pt} & & & &   \\
22& $\langle \hspace{.75pt} l, talt^{-1}, t^{-2}alt^2, t^4l^{-1}t^{-2}\rangle$ &28& RC&
$\frac{5+\sqrt{-3}}{2}$ \\
\vspace{-12pt} & & & &   \\
23& $\langle \hspace{.75pt} l, tal^{-1}t^{-1}, t^{-2}al^{-1}t^2, 
t^4lt^{-2} \hspace{.75pt}\rangle$&28& RC&
$\frac{-5+\sqrt{-3}}{2}$ \\
\vspace{-12pt} & & & &   \\
24& $\langle \hspace{.75pt} l, a, tl^{-1}t^{-1}a^{-1}t \hspace{.75pt}\rangle$
&28& RC& $\frac{5+\sqrt{-3}}{2}$ \\
\vspace{-12pt} & & & &   \\
25& $\langle \hspace{.75pt} l, a, tlt^{-1}a^{-1}t \hspace{.75pt}\rangle$
&28& RC& $\frac{-5+\sqrt{-3}}{2}$ \\
\vspace{-12pt} & & & &   \\
26& $\langle \hspace{.75pt} l, a, tal^{-1}t^{-1}, 
t^2alt^{-2}, t^4l^{-1}t^{-2} \hspace{.75pt}\rangle$&28& RC&
$\frac{5+\sqrt{-3}}{2}$ \\
\vspace{-12pt} & & & &   \\
27& $\langle \hspace{.75pt} l, a, t^2lt^{-2}, 
tl^{-1}at^{-2}a^{-1}lt^{-1} \hspace{.75pt}\rangle$&30& RD& $2$ \\
\vspace{-12pt} & & & &   \\
28& $\langle \hspace{.75pt} l, a, t^2lt^{-2}, 
tlat^{-2}a^{-1}l^{-1}t^{-1} \hspace{.75pt}\rangle$&30& RD& $2$ \\
\vspace{-12pt} & & & &   \\
29& $\langle \hspace{.75pt} a, tal^{-1}t^{-1}, 
t^{-1}al^{-1}t, t^{-1}lt^4l^{-1}t^{-1} \hspace{.75pt}\rangle$&30& RE&
$2$ \\
\vspace{-12pt} & & & &   \\
30& $\langle \hspace{.75pt} a, t^{-4}, ltl^{-1}t^{-1}, 
t^{-1}at^2a^{-1}t^{-1} \hspace{.75pt}\rangle$&30& RF& $2$ \\
\vspace{-12pt} & & & &   \\
31& $\langle \hspace{.75pt} a, t^{-4}, l^{-1}tlt^{-1},
t^{-1}at^2a^{-1}t^{-1} \hspace{.75pt}\rangle$&30& RF& $2$ \\
\vspace{-12pt} & & & &   \\
32& $\langle \hspace{.75pt} t, lt^{-2}l^{-1}, at^{-2}a^{-1},
 lat^2lt^{-1}a^{-1} \hspace{.75pt}\rangle$&30& RG& $2$
\\
\vspace{-12pt} & & & &   \\
33& $\langle \hspace{.75pt} t, a, lt^2at^2l, lt^{-6}l^{-1}, 
lt^{-2}at^{-2}l \hspace{.75pt}\rangle$&30& RD& $2$ \\
            \hline
    \end{tabular}
  \end{center}
\end{table}

For each $L_i$ we provide a reason related 
to $\Gamma$-action on the cosets of $L_i$ (column `Reason'), and 
a maximal ideal $I$ modulo which $L_i$ does not surject
onto  $\SL(2,\mathcal{O}_3/I)$ (column `Prime').
Structures of actions on cosets---that is, of the image of $\Gamma$
in $\mathrm{Sym}(|\Gamma:L_i|)$ under the permutation representation
arising from this action---are as follows (using
ATLAS~\cite{ATLAS} notation), with RA for `Reason A', etc. 
All factors $p^e$, $p$ prime, are minimal normal.
These structures cannot occur as quotients of any 
$\SL(2,\mathcal{O}/I)$. 

\medskip

RA: $2.2^4.2.2^4.\mathrm{A}_5$ (here $\SL(2,\mathcal{O}/I)$
 has quotients whose chief factor structure is the same, but 

they are not isomorphic to this group).

\vspace{3pt}

RB: $2^{12}.3^3.2^2.3$.

\vspace{3pt}

RC: $2^8\rtimes L_3(2)$.

\vspace{3pt}

RD: $3^5\rtimes \mathrm{A}_5$.

\vspace{3pt}

RE: $5^5\rtimes \mathrm{A}_5$.

\vspace{3pt}

RF: $\mathrm{A}_5^6.2.2^4.\mathrm{A}_5$.

\vspace{3pt}

RG: $3^{10}.2.2^4.\mathrm{A}_5$.

\medskip

The remaining $93$ subgroups of $\G$ of index at most $30$ 
are congruence subgroups. We verified this by showing that 
each subgroup contains a PCS of level $I$ for $I$ lying over various 
products of primes in $\{2,3,5,7,13,19\}$ (computed using 
 $\tt PrimesForDense$ and  $\tt PrimesForSubfields$); these 
being ideals $I$ modulo which some of the subgroups did not 
surject onto  $\SL(2,\mathcal{O}_3/I)$. 

\subsection{Finite-index normal subgroups of $\SL(2,\mathcal{O}_3)$}

In \cite{Alperin}, all normal subgroups of $\PSL(2, \mathcal{O}_3)$
of index less than $960$ are shown to be congruence subgroups; i.e., 
the preimage of each in $\G=\SL(2, \mathcal{O}_3)$ is a congruence 
subgroup. Using the ${\sf GAP}$ package \cite{LINS},
we calculated all $N\lhd \allowbreak \G$ of index up to $20000$ 
containing the center $\langle -1_2\rangle$ of $\Gamma$:
see Table~\ref{GammaMore}.

\begin{table}[h]
\begin{center}
\caption{Normal subgroups $N$ of $\Gamma$ of index less than $20000$}
\label{GammaMore}

\vspace{-10pt}

\begin{tiny}
\[
\begin{array}{c@{|}cccccc}
\mbox{Index}& 3& 12& 60& 168& 168& 180\\
\hline
\vspace{-6.5pt} & & & & & & \\
N&\Gamma^{(1)}&\Gamma^{(2)}&\Gamma(2)&\Gamma(2+\omega)
&\Gamma(3-\omega)&\Gamma^{(1)}\cap\Gamma(2)\\
\multicolumn{2}{c}{}\\
\mbox{Index} & 324& 504 & 504& 720& 960& 1092\\
\hline
\vspace{-6.5pt} & & & & & & \\
N &\Gamma^{(3)}=\Gamma(3)&\Gamma^{(1)}\cap\Gamma(2+\omega)
&\Gamma^{(1)}\cap\Gamma(3-\omega)&\Gamma^{(2)}\cap\Gamma(2)&H_1&\Gamma(3+\omega)\\
\multicolumn{2}{c}{}\\
\mbox{Index}&1092 & 1920& 2016& 2016 & 2880& 3276\\
\hline
\vspace{-6.5pt} & & & & & & \\
N
&\Gamma(1+3\omega)&\Gamma(4)&\Gamma^{(2)}\cap\Gamma(2+\omega)&\Gamma^{(2)}\cap\Gamma(3-\omega)
&\Gamma^{(1)}\cap\Gamma(2)\cap H_1& \Gamma^{(1)}\cap\Gamma(3+\omega)\\
\multicolumn{2}{c}{}\\\
\mbox{Index}
& 3276& 3420& 3420& 3840& 4032& 4032\\
\hline
\vspace{-6.5pt} & & & & & & \\
N
& \, \Gamma^{(1)}\cap\Gamma(1+3\omega)&\Gamma(3+2\omega)&\Gamma(2+3\omega)&H_2&H_3&H_4\\
\multicolumn{2}{c}{}\\\
\mbox{Index} & 5760& 7800& 8748& 10080& 10080& 10260\\
\hline
\vspace{-6.5pt} & & & & & & \\
N& \, \Gamma^{(1)}\cap\Gamma(4)&\Gamma(5)&\Gamma^{(4)}&\Gamma(2)\cap\Gamma(2+\omega)&
\Gamma(2)\cap\Gamma(3-\omega)&\Gamma^{(1)}\cap\Gamma(3+2\omega)
\\
\multicolumn{2}{c}{}\\\
\mbox{Index}
& 10260& 11520 & 11520& 13104& 13104& 14580\\
\hline
\vspace{-6.5pt} & & & & & & \\
N
&\Gamma^{(1)}\cap\Gamma(2+3\omega)&\Gamma^{(2)}\cap
H_1&\Gamma^{(1)}\cap\Gamma(4)\cap H_2&
\Gamma^{(2)}\cap\Gamma(3+\omega)&\Gamma^{(2)}\cap\Gamma(1+3\omega)&H_5\\
\multicolumn{2}{c}{}\\\
\mbox{Index} & 14880& 14880& 19440\\
\hline
\vspace{-6.5pt} & & & & & & \\
N
&\Gamma(5+\omega)&\Gamma(1+5\omega)&
\Gamma^{(2)}\cap\Gamma(2)\cap\Gamma(3)\\
\end{array}
\]
\end{tiny}
\end{center}
\end{table}

We use the notation of \cite{Alperin}; $\G(a)$ is the  
PCS of level $a\mathcal{O}_3$,
and $\G^{(i)}$ is the $i$th term of the derived series 
of $\G$. By happenstance $\Gamma(3)=\Gamma^{(3)}$, and thus 
$\G^{(i)}$ for $i\leq 3$ is a congruence subgroup.
Also, the intersection of two congruence subgroups is a 
congruence subgroup. The other normal subgroups $H_i$ are 
as below.

\begin{itemize}
\item
 $H_1= \big\langle \Gamma(4),
{\small \begin{bmatrix}
\omega-5\omega^2&8\omega^2\\
\vspace{-10pt}&\\
-4\omega&5\omega-\omega^2\end{bmatrix}}
\big\rangle$, i.e., $H_1/\Gamma(4)$
is the center of $\Gamma/\Gamma(4)$.
This is visibly a congruence subgroup.
\item $H_2$ is a subgroup of $\Gamma(4)$ of index $2$. 
It is the kernel of the action of $\Gamma$ on the cosets 
of
\[
\langle \hspace{.75pt} tlt, \hspace{.5pt} t^4, 
\hspace{.5pt} 
t^2at^{-2}, \hspace{.5pt} 
l^{-1}at^2a^{-1}l^{-1}, \hspace{.5pt}  
l^{-1}at^{-2}a^{-1}l^{-1},
\hspace{.5pt} 
tat^3a^{-1}t, 
\hspace{.5pt} 
 l^{-1}tat^{-2}a^{-1}tl^{-1}  \hspace{.75pt} \rangle .
\]
Modulo all prime ideals $I$ 
except $2\mathcal{O}_3$, $H_2$ 
surjects onto  $\SL(2,\mathcal{O}_3/I)$.
\item $H_3$ is a subgroup of 
$\Gamma^{(2)}\cap\Gamma(2+\omega)$
 of index $2$, and $H_4$ is a subgroup 
of the Galois conjugate
 $\Gamma^{(2)}\cap\Gamma(3-\omega)$
of index $2$. 
Neither $H_3$ nor $H_4$ surject modulo 
prime ideals dividing $3$ and $7$. Both contain
$\Gamma(7)\cap\Gamma(9)$,
and so they are congruence subgroups.
\item
$H_5$ is the kernel of the action by $\Gamma$ on the cosets 
of the subgroup $L_{27}$ (or  $L_{28}$ or  $L_{33}$) in 
Table~\ref{Tabe4}. Thus $H_5$ is not a congruence subgroup.
\end{itemize}

Succinctly: up to index $1000$, our list 
agrees with  that of \cite{Alperin}  (apart from 
$\Gamma^{(2)}\cap\Gamma(2)$, missing in \cite{Alperin});
whereas at larger indices, we readily discovered
 finite-index normal subgroups of $\G$ that are not 
congruence subgroups.

\subsection*{Acknowledgments}
We thank the International Centre for Mathematical Sciences
and Mathematisches Forschungsinstitut Oberwolfach 
 for hosting visits during which this paper was 
written. The third author has been supported in part by
Simons Foundation Grant~852063, which is gratefully 
acknowledged. The first and second authors are grateful for
the hospitality of the Department of Mathematics at 
Colorado State University.

\bibliographystyle{amsplain}

\begin{thebibliography}{10}

\bibitem{Alperin}
R.~C.~Alperin, Normal subgroups of 
$\mathrm{PSL}_2(\Z[\sqrt{-3}\hspace{.75pt}])$, 
\emph{Proc. Amer. Math. Soc.} \textbf{124} (1996), no.~10, 2935--2941. 

\bibitem{Aschbacher84}
M.~Aschbacher,
 On the maximal subgroups of the finite classical groups,
 {\em Invent. Math.} {\bf 76} (1984), no.~3, 469--514.

\bibitem{Brayetal}
J.~N.~Bray, D.~F.~Holt, and C.~M.~Roney-Dougal, 
\emph{The maximal subgroups of the low-dimensional finite classical 
groups}, London Math. Soc. Lecture Note Ser. \textbf{407}, 
Cambridge University Press, Cambridge, 2013.

\bibitem{Bridson}
M. R. Bridson, D. M. Evans, M. W. Liebeck, D. Segal,
Algorithms determining finite simple images of finitely presented groups,
 Invent. Math. \textbf{218} (2019), 623--648. 

\bibitem{ReidAnnMath}
M. R. Bridson, D. B. McReynolds, A. W. Reid, and R. Spitler,
Absolute profinite rigidity and hyperbolic geometry, Ann. of Math. (2) 
\textbf{192} (2020), no.~3, 679--719.

\bibitem{ATLAS}
J.~H.~Conway, R.~T.~Curtis, S.~P.~Norton, R.~A.~Parker, and
 R.~A.~Wilson. {\em {$\Bbb{ATLAS}$} of finite groups},
 Oxford University Press, Eynsham, 1985.

\bibitem{SLnO}
A.~S.~Detinko, D.~L.~Flannery, and A.~Hulpke, Computing congruence
quotients of Zariski dense matrix groups over number fields, preprint, 2026.

\bibitem{ExpMath}
A.~S.~Detinko, D.~L.~Flannery, and A.~Hulpke, 
Algorithms for experimenting with Zariski dense subgroups, 
\emph{Exp. Math.} \textbf{29} (2020), no. 3, 296--305.

\bibitem{SAT}
A.~S.~Detinko, D.~L.~Flannery, and A.~Hulpke, 
 The strong approximation theorem and computing with 
linear groups,  \emph{J.~Algebra} \textbf{529} (2019), 536--549.

\bibitem{Density}
A.~S.~Detinko, D.~L.~Flannery, and A.~Hulpke, Zariski density and
 computing in arithmetic groups, \emph{Math. Comp.} \textbf{87} 
(2018), no.~310, 967--986. 

\bibitem{Recog}
A.~S.~Detinko, D.~L.~Flannery, and E.~A.~O'Brien, 
 Recognizing finite matrix groups  over infinite fields, 
\emph{J. Symbolic Comput.} \textbf{50} (2013), 100--109.

\bibitem{TitsDFO}
A.~S.~Detinko, D.~L.~Flannery, and E.~A.~O'Brien,
 Algorithms for the Tits alternative and related problems,
{\em J. Algebra} \textbf{344} (2011), 397--406.

\bibitem{Fine}
B.~Fine,  \emph{Algebraic theory of the Bianchi groups}, 
Monogr. Textbooks Pure Appl. Math. \textbf{129},
Marcel Dekker, Inc., New York, 1989.

\bibitem{ZalesskiPreprint}
D.~L.~Flannery and A.~E.~Zalesski, 
Prime degree irreducible 
representations of simple algebraic groups and finite simple groups 
of Lie type, Arch. Math. (to appear).

\bibitem{Gap}
The~GAP Group, \emph{{GAP} -- {G}roups, {A}lgorithms, 
and {P}rogramming}, \url{http://www.gapsystem.org}.

\bibitem{LINS} The GAP package LINS,
\url{https://github.com/gap-packages/LINS}.

\bibitem{Handbook} 
D.~F.~Holt, B.~Eick, and E.~A.~O'Brien,
\emph{Handbook of computational group theory},
Chapman $\&$ Hall/CRC, Boca Raton, FL, 2005.

\bibitem{Hulpke} A.~Hulpke,
Proving infinite index for a subgroup of matrices.
 In {\em Computational {A}spects  of {D}iscrete {S}ubgroups
 of {L}ie {G}roups}, volume 783 of {\em Contemp. Math.}, 
	pages 83--85. American Mathematical Society, Providence,
	RI, 2023.

\bibitem{IsaacsCT}
I.~M.~Isaacs, {\em Character theory of finite groups},
Academic Press, New York-London, 1976.

\bibitem{Koch}
H.~Koch, \emph{Number theory, Algebraic numbers and functions},
Grad. Stud. Math. \textbf{24}, American Mathematical Society, 
Providence, RI, 2000.

\bibitem{Lubeck}  F. L\"ubeck, Small degree representations of 
finite Chevalley groups in defining characteristic,   
LMS J. Comput. Math. \textbf{4} (2001), 135--169.

\bibitem{Lubeck2}  F. L\"ubeck, Tables of weight multiplicities 
of small degree representations in defining characteristic, 
\url{https://www.math.rwth-aachen.de/~Frank.Luebeck/chev/WMSmall/index.html}

\bibitem{LubotzkySegal}
A.~Lubotzky and D.~Segal, \emph{Subgroup growth}, 
Birkh\"{a}user Verlag, Basel, 2003.

\bibitem{NSrecog}
M.~Neunh\"{o}ffer, \'{A}.~Seress, et al., The GAP package recog, 
\url{https://gap-packages.github.io/recog/}

\bibitem{RapinchukSAT}
A.~S.~Rapinchuk, Strong approximation for algebraic groups,
in {\em Thin groups and superstrong approximation}, 
volume~61 of {\em Math. Sci. Res. Inst. Publ.}, 
pages 269--298. Cambridge University Press, Cambridge, 2014.

\bibitem{Rivin}
I.~Rivin, Large Galois groups with applications to Zariski density,
\url{https://arxiv.org/abs/1312.3009}

\bibitem{SeitzZalesski}
G.~M. Seitz and A.~E. Zalesski,  On the minimal degrees of 
projective representations of the finite
  Chevalley groups, II.
 {\em J. Algebra} \textbf{158} (1993), 233--243.

\bibitem{Serre}
J.-P.~Serre,
Le probl\`{e}me des groupes de congruence pour $\mathrm{SL}_2$,
{\em Ann. of Math.} \textbf{92} (1970), no.~3, 489--527.

\bibitem{Swan}
R.~G.~Swan,  Generators and relations for certain special linear groups,
\emph{Advances in Math.} 
\textbf{6} (1971), 1--77.

\bibitem{Wehrfritz}
B.~A.~F.~Wehrfritz, \emph{Infinite linear groups}, 
Springer-Verlag, New York, 1973.

\bibitem{Weis}
B.~Weisfeiler,
Strong approximation for Zariski-dense subgroups of semisimple 
algebraic groups, \emph{Ann. of Math. (2)} \textbf{120} (1984), 
no.~2, 271--315.


\end{thebibliography}

\end{document}